\documentclass{article}

\usepackage{amssymb}
 \usepackage{amsthm}
\usepackage{amsmath}
 \usepackage{algorithmic}
 \usepackage{algorithm}
\usepackage[justification=centering]{caption}
\usepackage{ulem}
\usepackage{url}
\usepackage{hyperref}
\usepackage{breakurl}

\newtheorem{theo}{Theorem}

\newtheorem{cor}{Corollary}
\newtheorem{defi}{Definition}

\newtheorem{prop}{Proposition}

\begin{document}

\title{The Invere $p$-Maxian Problem on Trees with Variable Edge Lengths}
\author{Kien Trung Nguyen \\ 
              Mathematics Department, Teacher College, Cantho University, Vietnam. \\
             Email: trungkien@ctu.edu.vn }
\date{}
\maketitle

\begin{abstract}
We concern the problem of modifying the edge lengths of a tree in minimum total cost so that the prespecified $p$ vertices become the $p$-maxian with respect to the new edge lengths. This problem is called the inverse $p$-maxian problem on trees. \textbf{Gassner} proposed in 2008 an efficient combinatorial alogrithm to solve the inverse 1-maxian problem on trees. For the case $p \geq 2$, we claim that  the problem can be reduced to finitely many inverse $2$-maxian problems.  We then develop  algorithms to solve the inverse $2$-maxian problem for various objective functions. The problem under $l_1$-norm can be formulated as a linear program and thus can be solved in polynomial time. Particularly, if the underlying tree is a star, the problem can be solved in linear time. We also develop $O(n\log n)$ algorithms to solve the problems under Chebyshev norm and bottleneck Hamming distance, where $n$ is the number of vertices of the tree. Finally, the problem under weighted sum Hamming distance is $NP$-hard.\\

keywords: Location problem; $p$-Maxian; Inverse optimization; Tree.
\end{abstract}

\section{Introduction}
\label{sec0}

Location theory plays an important role in Operations Research due to its numerous applications. Here, we want to find optimal locations of new facilities. Location theory was intensively investigated, see Kariv and Hakimi $\cite{Kariv1,Kariv2}$, Hamacher $\cite{Hamacher}$, Eiselt $\cite{Eiselt}$. Recently, a new approach of location theory, the so-called inverse location problem, has been  focused and become an interesting research topic. In the inverse setting, we aim to modify the parameters in minimum cost so that the prespecified locations become optimal in the perturbed problem. I\textit{nverse median} and \textit{inverse center} problems are the popular topics in the field of inverse location theory.

For inverse 1-median problems, Burkard et al. $\cite{Burkard1}$ solved the inverse 1-median problem on trees  and the inverse 1-median problem on the plane with Manhattan norm in $O(n\log n)$ time. Then Galavii $\cite{Galavii}$ proposed a linear time algorithm  for the inverse 1-median problem on trees. Burkard et al. $\cite{Burkard0}$ investigated and solved the inverse Fermat-Weber problem   in $O(n\log n)$ time if the given points are not colinear.   Otherwise, the problem can be formulated as a convex program. For the inverse 1-median problem on a cycle, Burkard et al. $\cite{Burkard}$ developed an $O(n^{2})$ algorithm based on the concavity of the corresponding linear programmming constraints. Additionally, the inverse $p$-median problem on networks with variable edge lengths is $NP$-hard, see Bonab et al. $\cite{Bonab2}$. However, the inverse 2-median problem on a tree can be solved in polynomial time. More particularly, if the underlying tree is a star, the corresponding problem is solvable in linear time. Sepasian and Rahbarnia $\cite{Sepasian}$ investigated the inverse 1-median problem on trees with both  vertex weights and edge lengths variations. They proposed an $O(n\log n)$ algorithm to solve that problem.  While most recently papers concerned the inverse 1-median problem under linear cost functions, Guan and Zhang $\cite{Guan}$ solved the inverse 1-median problem on trees under Chebyshev norm ang Hamming distance by binary search algorithm in linear time.

Cai et al. $\cite{Cai}$ were the first who showed that although the 1-center problem on directed networks can be solved in polynomial time, the inverse problem is $NP$-hard. Hence, it is interesting to study some special situations which are polynomially solvable. Alizadeh and Burkard $\cite{Alizadeh1}$ developed a combinatorial algorithm with complexity of $O(n^{2})$ to solve the inverse 1-center problem on unweighted trees with variable edge lengths, provided that the edge lengths remain positive thoughout the modification. Dropping this condition, the problem can be solved in $O(n^{2}\textbf{c})$ time where $\textbf{c}$ is the compressed depth of the tree. For the corresponding uniform-cost problem, Alizadeh and Burkard $\cite{Alizadeh2}$ devised improved algorithms with running time $O(n\log n)$ and $O(\textbf{c}n\log n)$. Then Alizadeh et al. $\cite{Alizadeh3}$ use the AVL-tree structure to develop an $O(n\log n)$ algorithm for solving the inverse 1-center problem on trees with edge length augmentation. Especially, the uniform-cost problem can be solved in linear time. For the inverse 1-center problem on a simple generalization of tree graphs, the so-called cactus graphs, Nguyen ang Chassein $\cite{Nguyen}$ showed the $NP$-hardness. If the modification of vertex weights is taken into account, the inverse 1-center problem on trees can be solved in $O(n^{2})$ time, see Nguyen and Anh $\cite{Nguyen1}$. Furthermore, Nguyen and Sepasian $\cite{Nguyen2}$ solved the inverse 1-center problem on trees under Chebyshev norm and Hamming distance in $O(n\log n)$ time in the case there exists no topology change. In principle, the problem is solvable in quadratic time.
 
Although the inverse location problem was intensively studied, there is a limited number of  papers related to the inverse obnoxious location problem. Alizadeh and Burkard $\cite{Alizadeh4}$ developed linear time algorithm to solve the inverse obnoxious 1-center problem. Moreover, Gassner $\cite{Gassner}$ investigated the inverse 1-maxian problem on trees with variable edge lengths  and reduced the problem to a minimum cost circulation problem which can be solved in $O(n \log n)$ time. 

We focus in this paper the inverse $p$-maxian problem on trees with $p \geq 2$. This paper is organized as follows. We briefly introduce the optimality criterion of the $p$-maxian problem on trees as well as formulate the problem under arbitrary cost function in Section $\ref{sec1}$. Section $\ref{sec2}$ concentrates on the problem under $l_1$-norm. We show that the problem can be formulated as finitely many linear programs. If the underlying tree is a star, we can solve the problem in linear time. We focus on Section $\ref{sec3}$ and Section $\ref{sec4}$ the problems under Chebyshev norm and Hamming distance, respectively. We show that these problem  can be solved in $O(n\log n)$ time.

\section{The p-maxian problems on trees}
\label{sec1}
Given a graph $G = (V,E)$, $|V| = n$, each vertex $v_i \in V$ associates with a non-negative weight $w_i$ and each edge has a non-negative length $\ell_e$. The length of the shortest path $P(u,v)$ connecting two vertices $u$ and $v$ in $G$ is the distance $d(u,v)$ between these two vertices. A point on $G$ is either a vertex or lies on an edge of the graph. As proposed by Burkard et al. $\cite{Burkard2}$, the $p$-maxian problem on $G$ is to identify a set of $p$ points, say $X = \{x_1,x_2,\ldots,x_p\}$, to maximize the maxian objective function
\begin{center}
$ F(X) = \displaystyle \sum_{i=1}^{n} w_{i}\max_{1\leq j \leq p}d(v_{i},x_j)$.
\end{center}

By the vertex domination property, there exists a $p$-maxian of $G$ that is the set of $p$ vertices. Now we restrict ourselves in the case where the underlying graph $G$ is a tree, say $T$. Burkard et al. $\cite{Burkard2}$ already showed that $\{a,b\}$ is a 2-maxian of a tree  $T$ if $P(a,b)$ is its longest path. We further show that this condition is also the necessary condition as follows.
\begin{theo}\label{theoMaxian}(2-Maxian Criterion)\\
Given two vertices $a, b$ on the tree $T$, then $\{a,b\}$ is a 2-maxian of $T$ if and only if $P(a,b)$ is the longest path in the tree.
\end{theo}
\begin{proof}
For the sufficient condition, see $\cite{Burkard2}$. Assume that $P(a,b)$ is not the longest path of the tree, we will prove that $\{a,b\}$ is not its $2$-maxian as well. Let $m_{ab}$ be the midpoint of path $P(a,b)$. By deleting an edge that contains $m_{ab}$, we can deduce two parts $L$ and $R$ of $T$ which contain $a$ and $b$, respectively. Then we get
\begin{center}
$F(\{a,b\}) = \sum_{v \in L}w_vd(v,b) + \sum_{v \in R}w_vd(v,a)$.
\end{center}
Moreover, let $P(s,t)$ be the longest path of the tree. We trivially get $d(m_{ab},a) \leq \max\{d(m_{ab},s),d(m_{ab},t)\}$ and $d(m_{ab},b) \leq \max\{d(m_{ab},s),d(m_{ab},t)\}$, and at least one of the two inequalities does not hold with equality. Otherwise, it contradicts the assumption that $P(s,t)$ is the longest path. Therefore we get $d(v,a) \leq \max\{d(v,s),d(v,t)\}$ for $v \in R$ and $d(v,b) \leq \max\{d(v,s),d(v,t)\}$ for $v \in L$, and at least one of the inequalities will not hold with equality. We can finally conclude that $F(a,b) < F(s,t)$.
\end{proof}
 
We now reformulate the optimality criterion in Theorem $\ref{theoMaxian}$. Given two prespecified vertices $a$ and $b$, and a leaf  $v$, let $v_{ab}$ be the common vertex  of three paths $P(a,v)$, $P(b,v)$, and $P(a,b)$. The following result  states the necessary and sufficient conditions for  $\{a,b\}$ to be the longest path in $T$.
\begin{cor}\label{lemLong}(Longest path criterion)\\
$P(a,b)$ is the longest path of $T$ if and only if $d(v,v_{ab}) \leq d(a,v_{ab})$ and $d(v,v_{ab}) \leq d(b,v_{ab})$ for all leaves $v$ in $T$.
\end{cor}

Next we formulate the inverse version of the $p$-maxian problem on trees. Given a tree $T = (V,E)$ and a prespecified $p$-vertex. We can assume without loss of generality that the prespecified $p$ vertices are the leaves of $T$. The length of  each edge $e$ in  $E$ can be increased or decreased by an amount $p_e$ or $q_e$, i.e. the new length of $e$ is $\tilde{\ell}_e  = \ell_e + p_e -q_e$ and is assumed to be non-negative. We can state the inverse $p$-maxian on $T$ as follows.
\begin{enumerate}
\item The $p$-vertex becomes a $p$-maxian of the tree with respect to the new edge lengths $\tilde{\ell}$.
\item The cost function $\mathcal C(p,q)$ is minimized.
\item Modifications are feasible, i.e. $0 \leq p_e \leq \bar{p}_e$ and $0 \leq q_e \leq \bar{q}_e$.
\end{enumerate}

Burkard et al. $\cite{Burkard2}$ states that the set $S$ of $p$ vertices is a $p$-maxian of $T$ if  and only if it contains a pair of vertices $\{a,b\}$ such that $P(a,b)$ is the longest path of $T$. Therefore, the inverse $p$-maxian problem on a tree can be reduced to $p^{2}$ many $2$-maxian problems on this tree. From here on, we focus the 2-maxian problem on the tree $T$. 

Consider a monotone cost function $\mathcal C$ and $\{a,b\}$ is a pair of leaves which are the prespecified vertices, we get the following property.
\begin{prop}\label{propInCre}
In the optimal solution of the inverse $2$-maxian problem on $T$, it suffices to increase the lengths of edges in $P(a,b)$ and reduce the lengths of edges in $T\backslash P(a,b)$.
\end{prop} 
By Proposition $\ref{propInCre}$, we can set $x_e := p_e$ and $\bar{x}_e := \bar{p}_e$ for $e \in P(a,b)$, and $x_e := q_e$ and $\bar{x}_e = \bar{q}_e$ for $e \in T\backslash P(a,b)$. An edge $e$ is said to be modified by an amount $x_e$ if its modified length is set to $\tilde{\ell}_e := \ell_e + sign(e)x_e $, where $sign(e) = 1$ if $e \in P(a,b)$ and $sign(e) = -1$ if $e \not \in P(a,b)$. Furthermore, denote by $\mathcal L$ the set of leaves in the tree $T$ and $\tilde{d}$ by the distance measure with respect to new edge lengths $\tilde{\ell}$. In order to make $P(a,b)$ the longest path of $T$, it must hold that  $\tilde{d}(v,v_{ab}) \leq \tilde{d}(a,v_{ab})$ or $\tilde{d}(v,v_{ab}) \leq \tilde{d}(b,v_{ab})$ for all leaves $v$ in $T$, see Corollary $\ref{lemLong}$. After some elementary computations, we can formulate the inverse 1-maxian problem on $T$ as follows.
\begin{equation}\label{eqMax}
\begin{split}
\min \hspace{1.7cm} \mathcal C(x) \hspace{2cm}\\
\text{s.t. } \hspace{0.5cm}  \sum_{e \in P(a,v_{ab})}x_e + \sum_{e\in P(v,v_{ab})} &\geq \mathcal G(a,v), \hspace{0.5cm} \forall v \in \mathcal L\\
\sum_{e \in P(b,v_{ab})}x_e + \sum_{e\in P(v,v_{ab})} &\geq \mathcal G(b,v), \hspace{0.5cm} \forall v \in \mathcal L\\
0 \leq x_e &\leq \bar{x}_e, \hspace{1.3cm} \forall e\in E.
\end{split}
\end{equation}
Here, $\mathcal G(\star,v) := d(v,v_{ab}) - d(\star,v_{ab})$, $\star = a, b$, for $v \in \mathcal L$ are the gaps between distances. 
\section{The problem under $l_1$-norm}\label{sec2}
Assume that modifying an edge $e$ by a unit amount costs $c_e$, then the objective function in $\eqref{eqMax}$ under $l_1$-norm can be written as
\begin{center}
$\mathcal C(x) = \sum_{e\in E} c_ex_e$.
\end{center}
The inverse $2$-maxian problem on trees under $l_1$-norm can be formulated as a linear program. It is therefore solvable in polynomial time. However, an efficient combinatorial algorithm is still unknown.

We now consider the case where the underlying tree is a star graph with center vertex $v_0$.   We can directly deliver the following property from Corollary $\ref{lemLong}$.
\begin{cor}\label{Lem_StarMaxian}(Optimality criterion)\\
Given two leaf nodes $a$ and $b$ in a star graph $S$. Then, $\{a,b\}$ is a 2-maxian of the star graph if and only if $\ell_{(v_0,a)}$ and $ \ell_{(v_0,b)}$ are two largest edges.
\end{cor}

For simplicity, we denote $\ell_{(v_0,a)}$ and $ \ell_{(v_0,b)}$ by $\ell_a$ and $\ell_b$, and the cost to modify one unit length of $(v_0,a)$ and $(v_0,b)$ by $c_a$ and $c_b$. By Corollary $\ref{Lem_StarMaxian}$, we consider how to modify the length of egdes with minimum cost such that $\tilde{\ell}_a$ and $\tilde{\ell}_b$ become the two largest edges in the star. Asume that $\ell_a < \ell_b$, we analyze these two following cases.
\begin{enumerate}
\item If $\tilde{\ell}_a \in [\ell_a,\ell_b]$, we do not increase the length of $\ell_b$ but decrease the length of $e \neq (v_0,b)$ with $\ell_e > \tilde{\ell}_a$ to $\tilde{\ell}_a$. 
We first presolve the problem as follows.
\begin{itemize}
\item If $\xi =\displaystyle \min_{e\neq (v_0,b): \ell_e > \ell_{a} }\{\ell_e - \bar{x}_e\} > \ell_{a}$, then increase $\ell_{a}$ by an amount $\xi - \ell_{a}$.
\item If $\theta = \ell_a + \bar{x}_{(v_0,a)} < \ell_e$ for $e \neq (v_0,b)$, then we decrease the length of $\ell_e$ by an amount $\ell_e - \theta$.
\end{itemize}
If the presolution is not possible, then the problem is infeasible. Otherwise, we can formulate the problem as a univariate optimization program.
\begin{center}
$\min f(z) = c_a (z - \ell_{a}) + \displaystyle \sum_{e \neq (v_0,b): \ell_e > z} c_e (\ell_e - z)$ 
\end{center}
where $z = \tilde{\ell}_a$ and $z \in [\ell_a,\ell_a + \bar{x}_{(v_0,a)}]$.

\item  If $\tilde{\ell}_a > \ell_b$, then the length of $e$ where $\ell_e > \tilde{\ell_a}$ must be decrease to $\tilde{\ell}_a$ and the length of $\ell_b$ must be increase to $\tilde{\ell}_a$.
First, we have to increase $\ell_a$ an amount $\ell_b - \ell_a$ and presove the problem as in the first case. If the presolution is not possible then the problem is infeasible. Otherwise, we can fomulate the problem as follows.
\begin{center}
$\min f(z) = (c_a + c_b) (z - \ell_b) + \displaystyle \sum_{e: \ell_e > z} c_e (\ell_e - z)$ 
\end{center}
where $ z = \tilde{\ell}_a = \tilde{\ell}_b$ and $z \in [\ell_b,\min\{\ell_a +  \bar{x}_{(v_0,a)},\ell_b +  \bar{x}_{(v_0,b)}\}]$.
\end{enumerate}

In both cases, the function $f(z)$ is a convex function. According to Alizadeh and Burkard $\cite{Alizadeh4}$, the univariate optimization problem can be solve in linear time. To get the optimal solution of the inverse $2$-maxian problem on $S$, we have to solve two problems and get the best one. Therefore, we attain the following result.

\begin{theo}
The inverse  2-maxian problem on a star graph can be solved in linear time.
\end{theo}

\section{Problem under Chebyshev norm}\label{sec3}
We recall that modifying an edge $e$ by one unit length costs $c_e$, the objective function in $\eqref{eqMax}$ under Chebyshev norm can be written as
\begin{center}
$\mathcal C(x) = \max_{e\in E} \{c_ex_e\}$.
\end{center}

The solution method of the problem is based on the greedy modification. We define a valid  modification of egde lengths in $T$ with cost $C$ as follows.
\begin{defi}(Valid  modification, see $\cite{Nguyen2}$)\\
A valid  modification with cost $C$ is to modify each edge of $T$ as much as possible such that the cost is limited within $C$, i.e. we set 
\begin{center}
$ x_e := \begin{cases} \frac{C}{c_e} &\mbox{,if } c_e\bar{x}_e > C \\ \bar{x}_e & \mbox{,if } c_e\bar{x}_e \leq C. \end{cases}  $
\end{center}
\end{defi}

We can solve the problem by the following two phases.\\

\textbf{Phase 1:} \textit{Find the interval that contains the optimal cost.}\\

We first sort the costs $\{c_ex_e\}_{e \in E}$ in nondecreasing order and unite the similar values. Then we obtain a sequence of costs
\begin{center}
$c_1\bar{x}_1 < c_2\bar{x}_2 < \ldots < c_m\bar{x}_m$.
\end{center}  
Here, $m =O(n)$.


Then we aim to find the smallest index $i_0$ such that $P(a,b)$ become the longest path by applying the valid  modification with cost $c_{i_0}\bar{x}_{i_0}$. The index $i_0$ can be found by applying a binary search algorithm. If we apply the valid  modification with cost $c_k\bar{x}_k$ but the path $P(a,b)$ does not become the longest path then $i_0 > k$. Otherwise, we know $i_0 \leq k$. 

Let us analyze the complexity to find $i_0$. In each iteration modifying the edge lengths of the tree costs linear time. Then we find the longest path of the tree in linear time (see Handler $\cite{Handler}$), and compare the length of $P(a,b)$ with the longest one in order to decide if Corollary $\ref{lemLong}$ holds or not in linear time. As the binary search stops after $O(\log n)$ iterations, this procedure runs in $O(n\log n)$ time.

Assume that we have found the  interval $[c_{i_0-1}\bar{x}_{i_0-1},c_{i_0}\bar{x}_{i_0}]$ that contains the optimal cost. We apply a valid  modification with cost $c_{i_0-1}\bar{x}_{i_0-1}$ and get the modified tree $\tilde{T}$. Also, we update the upper bounds of modifications. The next step is to define a parameter $t \in [0,c_{i_0}\bar{x}_{i_0}]$ and find the smallest value $t$ such that $P(a,b)$ becomes the longest path of $\tilde{T}$ for a valid  modification with cost $c_{i_0}t$.\\

\textbf{Phase 2:} \textit{Find the minimizer $t$ with respect to the optimal objective value in $[0,c_{i_0}\bar{x}_{i_0}]$.}\\

After deleting all edges of the path $P(a,b)$ we get a forest $T^{del}$. Consider a vertex $v$ in the path $P(a,b)$, $v \neq a, b$. A branch $T_v$ is, by definition, a connected component of $T^{del}$ that contains the vertex $v$.   Denote by $\mathcal L(T_v)$  the set of leaves in $T_v$ except $a$ and $b$. Now we try to modify the edge lengths of $T$ so that there is no path $P(v,v')$, for $v \in P(a,b)$, $v \neq a, b$ and $v' \in \mathcal L(T_v)$, is longer than $\min\{d(v,a),d(v,b)\}$. To solve this problem, we apply a valid modification with cost $c_{i_0}t$ for $t \in [0,\bar{x}_{i_0}]$ and consider  the modifying lengths of these following paths.
\begin{itemize}
\item For a path $P(v,v')$ with $v' \in \mathcal L(T_v)$ and $d(v,v') > \min\{d(v,a),d(v,b)\}$, its modifying length is 
\begin{center}
$d(v,v') - \displaystyle c_{i_0}\sum_{e \in P(v,v'), \bar{x}_e > 0}\frac{t}{c_e}$.
\end{center}
\item For the paths $P(v,a)$ and  $P(v,b)$, we obtain the modifying lengths
\begin{center}
$d(v,a) + \displaystyle c_{i_0}\sum_{e \in P(v,a), \bar{x}_e > 0}\frac{t}{c_e}$,
\end{center}
and 
\begin{center}
$d(v,b) + \displaystyle c_{i_0}\sum_{e \in P(v,b), \bar{x}_e > 0}\frac{t}{c_e}$.
\end{center}
\end{itemize}
We first identify the value $t_1$ and $t_2$ such that
\begin{equation}\label{eqt1}
t_1 := \text{arg}\min \max\{d(v,v') - \displaystyle c_{i_0}\sum_{e \in P(v,v'), \bar{x}_e > 0}\frac{t}{c_e},d(v,a) + \displaystyle c_{i_0}\sum_{e \in P(v,a), \bar{x}_e > 0}\frac{t}{c_e}\},
\end{equation}
\begin{equation}\label{eqt2}
t_2 := \text{arg}\min \max\{d(v,v') - \displaystyle c_{i_0}\sum_{e \in P(v,v'), \bar{x}_e > 0}\frac{t}{c_e},d(v,b) + \displaystyle c_{i_0}\sum_{e \in P(v,b), \bar{x}_e > 0}\frac{t}{c_e}\},
\end{equation}
where $v \in P(a,b)$, $v' \in \mathcal L(T_v)$, $d(v,v') > \min\{d(v,a),d(v,b)\}$ and $t \in [0,\bar{x}_{i_0}]$. As we have to find the minimizer of the upper envelop of linear functions, the algorithm in Sec. 5 of Gassner $\cite{Gassner1}$ will find $t_1$ and $t_2$ in linear time provided that all linear functions have been found.

The smallest value $t^{*}$, such that the modifying lengths of all paths $P(v,v')$ for $v \in P(a,b)\backslash \{a,b\}$ and  $v' \in \mathcal L(T_v)$ are not larger than  $\min\{\tilde{d}(v,a),\tilde{d}(v,b)\}$, is easily identified as $t^{*} := \max\{t_1,t_2\}$.

We now analyze the complexity to find all linear functions in the expressions \eqref{eqt1}, \eqref{eqt2}. We number the vertices in the path $P(a,b)$ from the left to the right side by $a, v_1, v_2, \ldots ,v_k, b$, where $k = O(n)$. Denote by $E(S)$ the set of edges in a subtree $S$. We start from vertex $v_1$, it costs $O(|E(T_{v_1})|)$ time (by a breath-first-search procedure)  to identify all functions 
\begin{center}
$d(v_1,v') - \displaystyle c_{i_0}\sum_{e \in P(v_1,v'), \bar{x}_e > 0}\frac{t}{c_e}$ 
\end{center}
for $v' \in \mathcal L(T_{v_1})$ and $O(|E(P(a,b))|)$ time to identify 
\begin{center}
$d(v_1,\star) + \displaystyle c_{i_0}\sum_{e \in P(v_1,\star), \bar{x}_e > 0}\frac{t}{c_e}$ 
\end{center}
for $\star = a,b$.

Next we have to find all linear functions with respect to vertex $v_2$. To identify the functions
 \begin{center}
$d(v_2,\star) + \displaystyle c_{i_0}\sum_{e \in P(v_2,\star), \bar{x}_e > 0}\frac{t}{c_e}$ 
\end{center}
for $\star = a,b$, we have to exchange the modification edge $(v_1,v_2)$ that contribute to the augmentation of path $P(v_2,a)$. We continue the process for vertex $v_3, v_4, \ldots ,v_k$ similarly. 

In conclusion, it costs $O(\sum_{v \in P(a,b) \backslash\{a,b\}}|E(T_v)| + 2|E(P(a,b))|) = O(n)$ time to complete the procedure to find all linear functions in \eqref{eqt1} and \eqref{eqt2}. 

\begin{theo}
The inverse $2$-maxian problem on a tree under Chebyshev norm can be solved in $O(n\log n)$ time.
\end{theo}

\section{Problem under Hamming distance}\label{sec4}
First, let us focus on the problem under bottleneck Hamming distance. Assume that modifying one unit length of edge $e$ costs $c_e$, we aim to minimize the following objective function
\begin{center}
$\max_{e\in E}\{c_eH(x_e)\}$.
\end{center}
Here, $H$ is a Hamming distance and defined by
\begin{center}
$ H(\theta) := \begin{cases} 0 &\mbox{, if } \theta = 0 \\ 1 & \mbox{, otherwise} \end{cases}$
\end{center}

By the special structure of the Hamming distance, the objective function receives finitely many values, say $\{c_e: e\in E\}$. Therefore, we can solve the problem by finding the smallest value in $\{c_e: e\in E\}$ such that the optimality criterion in Lemma $\ref{lemLong}$ holds.

Number the edges in $T$ by $1,\ldots ,m$, and the corresponding costs are $c_1,\ldots ,c_m$, for $m=n-1$.  Let us first sort the costs $\{c_e: e\in E\}$ increasingly and get without loss of generality a sequence
\begin{center}
$c_1 \leq c_2 \leq \ldots \leq c_m$.
\end{center}

Now we apply a binary search algorithm to find the optimal cost. We start with the cost $c_k$, $k = \lfloor \frac{m+1}{2} \rfloor$. We modify all the edges $1,2,\ldots ,k-1$. If the optimality criterion holds, we know that the optimal value is less than or equal to $c_k$. Otherwise, it is larger than $c_k$. In each iteration we recomputing the length of the longest path in linear time (see Handler $\cite{Handler}$) and compare it with the length of $P(a,b)$. Moreover, as the binary search stop after $O(\log m)$ iterations, Phase 1 runs in $O(m \log m) = O(n\log n)$ time.

\begin{theo}
The inverse 2-maxian problem on trees under bottleneck Hamming distance can be solved in $O(n\log n)$ time.
\end{theo}

For the problem under weighted sum Hamming distance,  the objective function can be written as $\sum_{e \in E}c_eH(x_e)$. We can easily reduced the Knapsack problem into an inverse $p$-maxian problem under weighted sum Hamming distance in polynomial. Therefore, we get the following result.
\begin{theo}
The inverse $p$-maxian problem on tree under weighted sum Hamming distance is $NP$-hard.
\end{theo}

\section{Conslusion}
We have addressed the inverse $p$-maxian problem, $p\geq 2$, under various objective functions. It is shown that the problem can be reduced  to $p^{2}$ many $2$-maxian problems. Then we have formulated the inverse $2$-maxian problem on trees under $l_1$-norm as a linear program and solved the problem on star graphs in linear time. Furthermore,  the inverse $2$-maxian  problem under Chebyshev norm and Hamming distance is solvable in $O(n\log n)$ time, where $n$ is the number of vertices in the tree. For furture research topics, we will consider the inverse maxian problem on other classes of graphs, e.g., cacti, interval graphs, block graphs, etc.




\end{document}